\documentclass[11pt,reqno]{amsart}

\usepackage{amssymb,amsmath,amsthm,graphics,epsfig}
\usepackage{amsfonts}
\usepackage{graphicx,cite}
\usepackage{amssymb}
\usepackage[pagewise]{lineno}
\usepackage{mathrsfs}
\usepackage[colorlinks=true, pdfstartview=FitV, linkcolor=blue, citecolor=blue, urlcolor=blue]{hyperref}
\usepackage{booktabs}
\usepackage{graphicx}
\usepackage{tikz}
\usepackage{subfigure}
\usepackage{multicol}
\usepackage{cases}
\usepackage{booktabs}
\usepackage{orcidlink}
\usepackage{subcaption}
\usepackage{comment} 
\usepackage{arydshln}
\usepackage[colorinlistoftodos]{todonotes}
\usepackage[normalem]{ulem}
\usepackage{color}
\usepackage{amsaddr} 
\usepackage{float}
\usepackage[margin=2.7cm]{geometry}

\theoremstyle{definition}
\newtheorem{Lemma}{Lemma}
\newtheorem{Theorem}{Theorem}
\newtheorem{Proposition}{Proposition}
\newtheorem{Corollary}{Corollary}

\definecolor{ColorEdward}{rgb}{0.4,0.6,0.7}

\definecolor{ColorNolbert}{rgb}{0.2,0.5,0.1}

\usetikzlibrary{arrows.meta}

\usepackage{verbatim}

\def\R{\mathbb{R}}

\title[Extended Quasispecies with Time Lags and Fluctuations in replication]{Generalized Quasispecies Model with Time Delays and Periodic Fluctuations in Replication}

\author[N. Morales and E. Turner]{{Nolbert Morales$^{1}$\orcidlink{0009-0009-8114-1537} \MakeLowercase{and} Edward A. Turner$^{2}$\orcidlink{0000-0002-2959-9227}}}

\thanks{email: $^1$\texttt{nolbert.morales@uss.cl}, $^2$\texttt{edward.turner@uvm.cl}}

\address{$^{1}${\small Universidad San Sebasti\'an, Facultad de Ingenier\'ia,\\ Lago Panguipulli 1390, Puerto Montt, Chile}}
\address{$^{2}${\small Universidad Viña del Mar, Facultad de Ciencias Jurídicas, Sociales y de la Educación,\\ Viña del Mar, Chile}}

\subjclass{92D25, 47H11, 37N25}
\keywords{Modified quasispecies, delay differential equations, periodic solution, degree theory.}

\begin{document}
\begin{abstract}
In this research, we present a generalized quasispecies model in which population growth is governed by an arbitrary nonlinear function incorporating time delays.
We begin by demonstrating that, under the constant population constraint, the dynamics of the system with time delays remain confined to the invariant manifold for both forward and backward time evolution. Furthermore, we establish that in this modified quasispecies model, defined on a single-peak fitness landscape, in the presence of backward mutation and periodic fluctuations in replication rates and in replication probabilities,  the concentration of the \(i\)th replicating species exhibits a periodic behavior in time  independent of the magnitude of the time delays. Specifically, this concentration oscillates between the minimum and maximum values of the probabilities \(Q_{ji}\) associated with erroneous replication; that is, the probability that a mutated replicator of type \(j\) produces an offspring of type \(i\). Moreover, under the presence of time delays and non-constant periodic fluctuations in replication rates, we show that if the probability that a mutated replicator of type \(j\) produces an offspring of type \(i\) remains constant across all replicators, then the unique positive periodic solution is necessarily a constant solution.
\end{abstract}

\maketitle

\section{Introduction}

The quasispecies framework was introduced in the 1970's by Manfred Eigen~\cite{Eigen1971} and later developed in collaboration with Peter Schuster~\cite{Eigen1988}, with the aim of describing the behavior of replicating entities under high mutation rates. Originally formulated to address questions related to prebiotic evolution, the theory has since been broadened to encompass diverse biological systems exhibiting elevated mutational dynamics. These include RNA viruses~\cite{Mas2004,Perales2020,Revull2021}, as well as cancer cells characterized by pronounced genetic instability~\cite{Sole2003,Sole2004,Brumera2006}.

Empirical research on quasispecies dynamics have predominantly centered on RNA viruses infecting bacteria, animals, and plants. The relevance of these studies to RNA-based genetics was extensively documented in the pioneering reviews by Domingo and colleagues~\cite{Domingo1985,Domingo1988}, which provided a foundational understanding of the experimental underpinnings of the theory.  

The presence of a mutant spectrum in the nature of RNA viruses was first confirmed through clonal sequencing of bacteriophage Q$\beta$, which revealed significant genetic diversity even when the infection started from a single virion~\cite{Domingo1978}. Since then, genetically diverse viral populations have been observed in many RNA viruses, including foot-and-mouth disease virus~\cite{Domingo1980,Sobrino1983}, vesicular stomatitis virus~\cite{Holland1979,Holland1982}, hepatitis B and C viruses~\cite{Martell1992,Davis1999,Mas2004,Perales2020} and others. These findings confirm that high genetic variability within viral populations is a common feature of RNA virus replication.

Beyond virology, the classical quasispecies model relies on simplifying assumptions: infinite population size, uniform replication rates, linear growth dynamics, and static environments. Modern extensions have sought to introduce more realism by incorporating finite population effects~\cite{Nowak1989,Sardanyes2008}, stochasticity~\cite{Sardanyes2011,Ari2016}, spatial structure, recombination, and complementation~\cite{Sardanyes2010}. These models reveal complex phenomena absent from the original formulation, such as the ``survival-of-the-flattest'' effect, in which mutational robustness can outweigh raw replication rate under high error loads~\cite{Wilke2001,Wilke2001b}.

Despite significant progress in the development of quasispecies models, certain biologically relevant factors remain underexplored. For instance, periodic replication rates—driven by circadian rhythms or environmental fluctuations such as temperature changes—are especially pertinent in plant systems~\cite{Obreepalska2015,Honjo2020}. Moreover, traditional quasispecies theory and most RNA virus replication models typically assume instantaneous genome synthesis. However, recent experimental findings indicate that the SARS-CoV-2 replication machinery operates with an elongation speed of approximately 150 to 200 nucleotides per second~\cite{Campagnola2022}. This rate is more than twice that of the poliovirus RNA polymerase. Based on these estimates, a complete SARS-CoV-2 genome of roughly 29.9 kilobases can be synthesized within 2.5 to 3.32 minutes, whereas the poliovirus genome (approximately 7.4 kilobases) is replicated in about 1.5 minutes.
In this work, we present a generalized quasispecies model that incorporates biologically motivated extensions: arbitrary nonlinear growth, time delays, and periodic replication rates.
This model is defined over a single-peak fitness landscape and governed by a system of delay differential equations. 
We analyze its dynamics under a constant population constraint, showing that the system's trajectories remain confined to an invariant manifold in both forward and backward time. Furthermore, by applying tools from the theory of functional differential equations and topological degree theory (specifically the Leray–Schauder degree), we demonstrate the existence of periodic solutions even in the presence of backward mutation and regardless of the magnitude of the delays. These findings extend the classical quasispecies paradigm and highlight how realistic temporal structure influences evolutionary dynamics.

This work is structured as follows. In Section~\ref{sec:model}, we formulate a generalized quasispecies model describing the dynamics of $n$ interacting populations under nonlinear growth, time delays, and periodic replication rates. We show that the classical constant population (CP) constraint defines an invariant manifold for the delayed system.
Section~\ref{sec:results} we consider a quasispecies model defined on a single-peak fitness landscapes $(n=1)$ and establish the existence of periodic solutions under backward mutation and time-periodic replication, regardless of the magnitude of the delay.
In Section~\ref{sec:applications}, we apply the main Theorem in two illustrative two-dimensional systems (constant and exponential growth), obtaining explicit conditions for the existence of periodic solutions.

\section{A generalized quasispecies model with time lags and periodic fitness}\label{sec:model}
While the classical quasispecies model has provided foundational insights into molecular evolution, it assumes constant mutation and replication rates, and overlooks key factors such as growth limitations and environmental variability. These simplifications limit its applicability over extended time frames. In this section, we introduce a modified version of the model that incorporates time delays, periodic fluctuations in replication rates, and the nonlinearity of population growth, thereby offering a more realistic framework for capturing evolutionary dynamics under temporally varying conditions.

\subsection{On the quasispecies model} 
In~\cite{TurnerQuasi1}, the authors extend the classical quasispecies framework by incorporating the effects of time delays and periodic fluctuations in the replication process of RNA genomes, under the assumption of constant population (CP). These modifications are designed to capture biologically relevant features such as the temporal separation between genome synthesis and other stages of the replication cycle. The resulting model is described by the following non-autonomous delay differential system:
\begin{equation}\label{eq:QuasiDelay}
	\dot x_{i}(t)=\sum_{j=0}^n f_j(t)Q_{ji}x_j(t-\tau_j)-x_{i}(t)\Phi(\mathbf x), 
\end{equation}
where $x_i(t)$ denotes the concentration of the $i$-th replicator species. The function $f_j(t)$ represents the time-periodic replication rate of the $j$-th species, and $\tau_j$ introduces a discrete time delay specific to each replicator type. The mutation transition matrix $Q_{ji}$ encodes the probability that replication of type $j$ yields an offspring of type $i$. The outflow term is defined as $\Phi(\mathbf{x}) = \sum_{j=0}^n f_j(t)x_j(t - \tau_j)$, where $\mathbf{x} = (x_0, \dots, x_n)$ is the population state vector.
The authors showed that de CP condition still remains valid with time lags and restrict the domain in
\begin{equation}\label{eq:CP}
\Sigma_{n+1}=\left\{ \mathbf x\in\mathbb R^{n+1}:x_0+x_1+\dots+x_n=1,\ x_0,\dots x_n\geq0 \right\}.
\end{equation}

In~\cite{Zapata:2022aa}, the authors propose a modified quasispecies model in which the growth rate of each population decreases as it approaches the maximum value $x_i = 1$. To capture this logistic-like regulation, the constant replication rate $f_j x_j$ is replaced by the density-dependent term $f_jx_j(1 - x_j)$. This modification is both biologically meaningful and mathematically natural, as it accounts for resource limitations and saturation effects commonly observed in real populations—features often overlooked in classical formulations.

\subsection{A general quasispecies model}
In this article, we propose a dynamic model that extends the classical quasispecies framework by introducing a broader class of nonlinearities in the replication rate. Specifically, we incorporate a functional form of the type $f_j(t) \, \psi(x_j(t - \tau_j(t)))$, where $\psi\in C([0,1],[0,1])$ such that $\psi(x_0)=0$ if and only if $x_0=0$ or $x_0=1$, as illustrated in Figure~\ref{fig:1}. Moreover, the model accounts for both time delays and periodic fluctuations in the replication dynamics, which are essential features in various biological contexts.

The proposed model is described by the following non-autonomous system of delay differential equations:
\begin{equation}\label{eq:SistPrincipal}
\dot{x}_i = \sum_{j=0}^n f_j(t) Q_{ji} (t)\psi(x_j(t - \tau_j(t))) - x_i(t) \Phi_{n+1}(\mathbf{x}),
\end{equation}
where $\tau_{j}$, $f_j$ are $T$-periodic continuous functions, $\tau_{j}(t)\geq 0$ and $f_j(t)>0$ for all $t\in\R$, the function $\Phi_{n+1}(\mathbf{x})=\sum_{j=0}^n f_j(t) \psi(x_j(t - \tau_j(t))).$
It is worth noting that if we choose $\psi(x_j) = x_j$, system~\eqref{eq:SistPrincipal} reduces to the classical quasispecies model \cite{TurnerQuasi1}, as given in \eqref{eq:QuasiDelay}. Similarly, selecting $\psi(x_j) = x_j(1 - x_j)$ produces a regulated quasispecies model with logistic growth constraints \cite{Zapata:2022aa}.

\begin{figure}
\begin{center}
\begin{tikzpicture}[scale=1]
	\node[inner sep=0pt] (caso1) at (-4.3,0) {\includegraphics[scale=0.6]{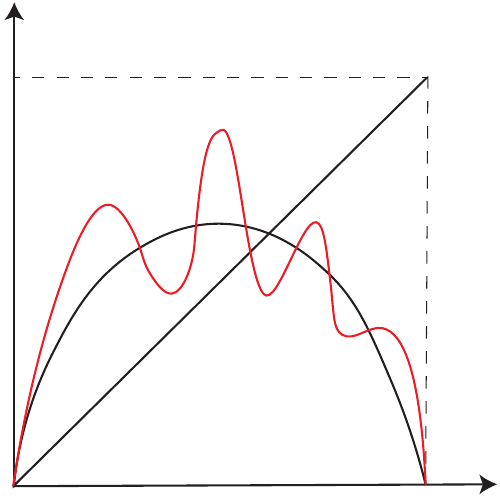}};
	\coordinate [label=below:\textcolor{black}{\small$\psi(x_i)$}] (E2) at (-7.3,2.8);
    \coordinate [label=below:\textcolor{black}{\footnotesize$\psi(x_i)=x_i$}] (E2) at (-1.5,2.1);
    \coordinate [label=below:\textcolor{black}{\footnotesize$\psi(x_i)=(1-x_i)x_i$}] (E2) at (-1,-0.2);
    \coordinate [label=below:\textcolor{black}{\small$x_i$}] (E2) at (-1.8,-2.5);
    \coordinate [label=below:\textcolor{black}{\footnotesize$0$}] (E2) at (-6.9,-2.5);
    \coordinate [label=below:\textcolor{black}{\footnotesize$1$}] (E2) at (-2.5,-2.5);
    \coordinate [label=below:\textcolor{black}{\footnotesize$1$}] (E2) at (-7,2);
\end{tikzpicture}
    \end{center}
   \captionsetup{width=\linewidth}
	\caption{The red curve depicts the nonlinear behavior of the replication ratio represented by $\psi$, while the black curves correspond to the reference cases $\psi = x_i$ and $\psi = (1 - x_i)x_i$.}
	\label{fig:1}
\end{figure}
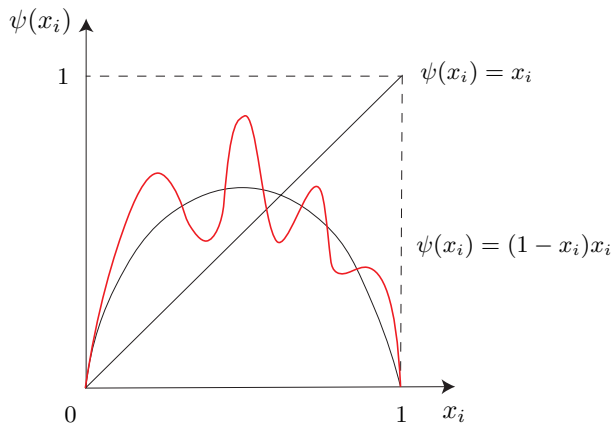

As previously discussed, the CP condition plays a central role in governing the behavior of the quasispecies framework, particularly as it relates to the invariant manifold defined by~\eqref{eq:CP}. In the following analysis, we establish that this condition remains robust even when the system incorporates temporal delays and includes the term $\psi(x_j(t-\tau_j(t)))$. To substantiate this, our approach is twofold. First, Proposition~\ref{vari_inva} confirms that if the right endpoints of the initial function lie in the set $\Sigma_{n+1}$, the resulting trajectory will persist in $\Sigma_{n+1}$ for all future times. This implies that full inclusion of the initial function within $\Sigma_{n+1}$ is not a prerequisite. Subsequently, Proposition~\ref{SolPe} shows that any $T$-periodic solution which maintains strict positivity throughout its domain must be entirely contained within $\Sigma_{n+1}$.

\begin{Proposition}\label{vari_inva} Let $\mathbf x(t)$ be the solution of a system \eqref{eq:SistPrincipal} with initial conditions
$$\mathbf x(t)=\phi(t)=(\phi_0(t),\phi_1(t),\dotsc,\phi_n(t)),\quad \forall t\in\left[t_0-\gamma_n,t_0\right]$$
where $\gamma_n= \displaystyle\max_{t\in[0,T]}\{\tau_{j}(t):i=0,n\}$ and such that $\mathbf \phi(t_0)\in\Sigma_{n+1}$. Then $\mathbf x(t)\in\Sigma_{n+1}$ for all $t\in[t_0,+\infty[$ .
\end{Proposition}
 
\begin{proof}
Let $ \mathbf x(t) = (x_0(t), x_1(t), \dotsc, x_n(t)) \in C^1([t_0, +\infty[, \mathbb{R}^n) $ be a solution of system~\eqref{eq:SistPrincipal}. Observe that  
$$
\frac{d}{dt}\left( \sum_{i=0}^n x_i(t) \right) = \left(1 - \sum_{i=0}^n x_i(t) \right) \sum_{j=0}^n f_j(t)\psi(x_j(t - \tau_j(t))).
$$

It is straightforward to verify that if there exists $ t_1 \geq t_0 $ such that $x_i(t_1) = 1$ or $x_i(t_1) = 0$, then $\dot x_i(t_1) \leq 0$ or $\dot x_i(t_1) \geq 0$, respectively. Consequently, it follows that $0 \leq x_i(t) \leq 1$ for all $i$ and for all $t \geq t_0 $.
Now, consider $z(t) = \sum_{i=0}^n x_i(t).$ Using the method of steps applied to system~\eqref{eq:SistPrincipal}, the system reduces, for $t \in [t_0, t_0 + \gamma_n]$, to the ordinary differential equation:
\begin{equation}\label{EDO}
    \dot z(t) = (1 - z(t)) \sum_{j=0}^n f_j(t)\psi(\phi_j(t - \tau_j(t))),
\end{equation}
with initial condition $z(t_0) = 1.$ By the existence and uniqueness theorem, it follows that $z(t) = 1$ for all $t \in [t_0, t_0 + \gamma_n]$.
Applying the method of steps again for $t \in [t_0 + \gamma_n, t_0 + 2\gamma_n],$ we obtain the same differential equation~\eqref{EDO} with initial condition $z(t_0 + \gamma_n) = 1,$ leading again to $z(t) = 1$ on this interval.
Continuing this process inductively, we conclude that $z(t) = \sum_{i=0}^n x_i(t) = 1$ for all $t \geq t_0.$
\end{proof}

\begin{Proposition}\label{SolPe} Let $\mathbf x(t)=(x_0(t),x_1(t),\dotsc,x_n(t))$ be a $T$-periodic solution of a system \eqref{eq:SistPrincipal} such that $x_i(t)\in]0,1[$ for all $t\in[0,T]$. Then $\mathbf x(t)\in\Sigma_{n+1}$ for all $t\in[0,T]$.
\end{Proposition}

\begin{proof}
Let $\mathbf{x}(t) = (x_0(t), x_1(t), \dots, x_n(t))  $ be a $T$-periodic solution of \eqref{eq:SistPrincipal}. Then, we have
$$
\frac{d}{dt} \left( \sum_{i=0}^n x_i(t) \right) = \left(1 - \sum_{i=0}^n x_i(t)\right) \sum_{j=0}^n f_j(t) \psi(x_j(t - \tau_j(t))).
$$
Since $\sum_{i=0}^n x_i(t)$ is also a $T$-periodic function, there exists $t_0 \in \mathbb{R}$ such that
$$
0 = \left(1 - \sum_{i=0}^n x_i(t_0)\right) \sum_{j=0}^n f_j(t_0) \psi(x_j(t_0 - \tau_j(t_0))).
$$
Hence, $\sum_{i=0}^n x_i(t_0) = 1$ i.e, $\mathbf x(t_0)\in\Sigma_{n+1}$. By Proposition \ref{vari_inva}, it follows that $\mathbf{x}(t) \in \Sigma_{n+1}$ for all $t \in [t_0, +\infty[$. Finally, since $\mathbf{x}$ is $T$-periodic, we conclude that $\mathbf{x}(t) \in \Sigma_{n+1}$ for all $t \in \mathbb{R}$.
\end{proof}

In light of Proposition \ref{SolPe}, it is sufficient to ensure the existence of $T$-periodic solutions within the open set $]0,1[^{n+1}$ in order to guarantee $T$-periodic solutions of system \eqref{eq:SistPrincipal} in the simplex $\Sigma_{n+1}$. This result, together with the preceding analysis, provides a solid foundation for the biological interpretation of the model, both in the presence and absence of time delays.

\section{Modified quasispecies model for the single-peak fitness landscape}\label{sec:results}
In this section, we present a simplified version of the modified quasispecies model, assuming a minimal fitness landscape. Under this assumption, the population is represented by two classes: the master sequence, denoted by $x_0$, characterized by superior replication efficiency, and a collective mutant class, denoted by $x_1$, which aggregates all deleterious variants with reduced fitness. This configuration corresponds to a binary fitness landscape often referred to in the literature as the Swetina-Schuster landscape~\cite{Swetina1982,Sole2003}. This landscape it facilitates analytical treatment due to its reduced dimensionality, especially under the constant population (CP) assumption, which allows for a reduction in the number of independent variables; and it has been shown to reproduce key features of quasispecies structure observed in empirical viral data~\cite{Sole2006}. The specific model considered here corresponds to the case $n = 1$ in system~\eqref{eq:SistPrincipal}.
In what follows, we employ topological methods based on the Leray–Schauder degree to investigate the existence of at least one positive $T$-periodic solution.

\subsection{Notation, conditions and main result}
As usual, we write $C_T:=\left\{\theta\in C([0,T],\R):\right.$ $ \left.\theta(0)=\theta(T)\right\}.$ For $\theta_1$ and $\theta_2$ in $C_T$ we define $\overline{\theta}_1=\frac{1}{T}\int_0^T\theta_1(t)dt$ and $Q[\theta_1,\theta_2]=(\overline{\theta_1},\overline{\theta_2}),$ $P[\theta_1,\theta_2]=(\theta_1(0),\theta_2(0))$ and $K[\theta_1,\theta_2]=\left(\int_0^t\theta_1(s)ds,\int_0^t\theta_2(s)ds\right).$ In addition, we denote $q=Q_{00},$ $p=Q_{10}$ and
\begin{equation}
     \nonumber \eta=\min_{t\in[0,T]}\{p(t),q(t)\},\quad \nu=\max_{t\in[0,T]}\{p(t),q(t)\}.
\end{equation}
    Finally, we consider $\mathcal{N}:[0,1]\times[0,1]\to\R$ such that 
\begin{align*}
           \mathcal{N}(x,y)=&(\overline{qf_0}\psi'(x)-\overline{f_0}(\psi(x)+x\psi'(x))-\overline{f_1}\psi(y))((\overline{f_1}-\overline{pf_1})\psi'(y)-\overline{f_1}(\psi(y)+y\psi'(x))-\overline{f_0}\psi(x))\\
           &\qquad-(\overline{f_0}-\overline{qf_0}-\overline{f_0}y)(\overline{pf_1}-\overline{f_1}x)\psi'(x)\psi'(y).
\end{align*}    
In order to formulate our result precisely, we introduce the following additional assumptions:\\
\textbf{[S]}:  The  system of nonlinear real equations
\begin{equation}\label{sist}
\begin{array}{rc}
     \overline{qf_0}\psi(x)+\overline{pf_1}\psi(y)-x(\overline{f_0}\psi(x)+\overline{f_1}\psi(y))&=0,  \\
     (\overline{f_0}-\overline{qf_0})\psi(x)+(\overline{f_1}-\overline{pf_1})\psi(y)-y(\overline{f_0}\psi(x)+\overline{f_1}\psi(y))&=0, 
\end{array}
\end{equation}
admits at least one solution in the open square $]0,1[\times]0,1[$.\\
\textbf{[C]}: $\mathcal{N}(x_i,y_i))\neq 0$ and $\sum_{i=1}^k\mbox{sign}({\mathcal{N}(x_i,y_i)})\neq 0$, where $(x_i,y_i)$, $i=1,k$  are solutions of the system \eqref{sist} in $]0,1[\times]0,1[$.

\begin{Theorem}\label{TeoP}   Let us assume that  \textbf{[C]},    \textbf{[S]}  are satisfied and $0<\eta<\nu<1$.  Then  \eqref{eq:SistPrincipal} admits at least one $T$-periodic positive solution  such that\begin{equation*}
     \eta\leq x_0(t)\leq \nu,\quad 1-\nu\leq x_1(t)\leq1-\eta \quad \mbox{ and } \quad x_0(t)+x_1(t)=1 \mbox{ for all }t\in[0,T].
\end{equation*}
\end{Theorem}

\begin{Theorem}\label{TeoS}  Suppose that $0<\eta=\nu<1$. Then equation  \eqref{eq:SistPrincipal} possesses only one positive $T$-periodic solution, which is the constant solution $x_0=x_1=\eta$. 
\end{Theorem}

\subsection{Abstract setting}
We now reformulate the problem of finding $T$-periodic solutions of system \eqref{eq:SistPrincipal} as a fixed-point problem for a relatively compact operator. Let us define
\begin{equation*}\mathcal{A} :C_T\times C_T\to C_T\times C_T,\quad \mathcal{A}[\mathbf{x}]= \mathcal{A} [x_0,x_1] =P[x_0,x_1] +QN[x_0,x_1] +K(I-Q)N[x_0,x_1],
\end{equation*}
where $$N[x_0,x_1] (t)=  
\left(\begin{array}{lr}  \displaystyle\sum_{j=0}^1f_j(t)Q_{j0}(t)\psi(x_j(t-\tau_{j}(t)))-x_0(t)\Phi_{2}(\mathbf{x})\\  \displaystyle\sum_{j=0}^1f_j(t)Q_{j1}(t)\psi(x_j(t-\tau_{j}(t)))-x_1(t)\Phi_{2}(\mathbf{x})\end{array}\right).$$
We now consider the operator $\mathcal{A} : C_T \times C_T \to C_T \times C_T$. Under this formulation, $\mathcal{A}$ is completely continuous, and hence the Leray--Schauder degree theory is applicable (see, e.g., \cite{Mawhin}). In this framework, the periodic boundary value problem associated with \eqref{eq:SistPrincipal} becomes equivalent to the fixed-point problem
\begin{equation*}\label{opecu}
(x_0, x_1) = \mathcal{A}[x_0, x_1], \qquad (x_0, x_1) \in C_T \times C_T.
\end{equation*}
The key step in proving Theorem \ref{TeoP} is to show that the topological degree of the operator $I - \mathcal{A}$ on a suitable open subset of $C_T \times C_T$ is nonzero.
To this end, for sufficiently small $\epsilon > 0$, we define the open bounded set
$$
\Omega := \left\{(x_0, x_1) \in C_T \times C_T : \eta - \epsilon < x_0(t) < \nu + \epsilon, \;\; 1 - \nu - \epsilon < x_1(t) < 1 - \eta + \epsilon, \;\; \forall t \in [0, T] \right\}.
$$
Next, we introduce the following homotopy:
$$
\mathcal{H} : [0,1] \times C_T \times C_T \to C_T \times C_T,
$$
$$
\mathcal{H}[\lambda, x_0, x_1] := (x_0, x_1) - \left(P[x_0, x_1] + QN[x_0, x_1] + \lambda K(I - Q)N[x_0, x_1]\right),
$$
where $P$, $Q$, $K$, and $N$ denote appropriate linear and nonlinear operators as defined in the context.
Before proceeding with the detailed analysis, we define $\Sigma$ as the set of possible solutions of the homotopy equation $\mathcal{H}[\lambda, x_0, x_1] = 0$ for $\lambda \in [0,1]$. Our first objective will be to prove that $\Sigma \subseteq \Omega$.

\begin{Lemma}\label{Fromt}
Let $0 < \eta < \nu < 1$, and suppose that $(x, y) \in \,]0,1[ \times \,]0,1[$ is a solution to the system \eqref{sist}. Then $(x, y) \in [\eta, \nu] \times [1 - \nu, 1 - \eta].$
\end{Lemma}
\begin{proof}
Let $0 < \eta < \nu < 1$, and suppose that $(x, y) \in\, ]0,1[ \times \,]0,1[$ is a solution of system \eqref{sist}. 
Assume first that $x \in\, ]0, \eta[$. Then, from system \eqref{sist}, we obtain
$$0 < \frac{\psi(y)}{T} \int_0^T f_1(s)(p(s) - x)\, ds = \frac{\psi(x)}{T} \int_0^T f_0(s)(x - q(s))\, ds < 0,$$
which is a contradiction. Similarly, if $x \in\, ]\nu, 1[$, then
$$ 0 > \frac{\psi(y)}{T} \int_0^T f_1(s)(p(s) - x)\, ds = \frac{\psi(x)}{T} \int_0^T f_0(s)(x - q(s))\, ds > 0, $$
also a contradiction.
Now assume that $y \in\, ]0, 1 - \nu[$ or $x\in]1-\eta,1[$, then from \eqref{sist}  we obtain respectively that
     \begin{align*}
     0<\frac{\psi( y)}{T}\int_0^Tf_1(s)(1-y-p(s))ds&= \frac{\psi(x )}{T}\int_0^Tf_0(s)(q(s)-1+y)ds<0,\\
         0>\frac{\psi(y)}{T}\int_0^Tf_1(s)(1-y-p(s))ds&= \frac{\psi(x )}{T}\int_0^Tf_0(s)(q(s)-1+y)ds>0.         
     \end{align*}
     This implies a contradiction in both cases.
Therefore, none of the values outside $[\eta, \nu] \times [1 - \nu, 1 - \eta]$ can be attained by a solution $(x, y)$ of the system, which completes the proof.
\end{proof}

\begin{Lemma}\label{cots1} Assume that all the hypotheses of Theorem \ref{TeoP} are verified. Then $\Sigma\subseteq \Omega$ for every $\lambda\in [0,1]$ for $t\in [0,T].$
\end{Lemma}

\begin{proof}
Suppose, that there exist $\lambda \in [0,1]$ and $(x_0, x_1) \in \Sigma$ such that $\theta \in \partial \Omega$. Consider first the case $\lambda \in ]0,1]$. From the definition of homotopy $\mathcal{H}$, we have the following system:
\begin{equation}\label{ecT1}
\begin{array}{lr}\dot x_0(t)=\lambda\displaystyle\sum_{j=0}^1f_j(t)Q_{j0}(t)\psi(x_j(t-\tau_{j}(t)))-x_0(t)\Phi_{2}(\mathbf{x})\\   
\dot x_1(t)= \lambda\displaystyle\sum_{j=0}^1f_j(t)Q_{j1}(t)\psi(x_j(t-\tau_{j}(t)))-x_1(t)\Phi_{2}(\mathbf{x}),\end{array}
\end{equation}
Assume there exist $t_x, t_y \in [0,T]$ such that $x_0'(t_x) = 0$ and $x_1'(t_y) = 0$. Evaluating \eqref{ecT1} at these times yields:
 \begin{equation}\label{x'}
  \psi( x_1(t_x-\tau_{1}(t_x)))(p(t_x)-x_0(t_x))=r(t_x)\psi(x_0(t_x-\tau_{0}(t_x)))(x_0(t_x)-q(t_x)),
\end{equation}
\begin{equation}\label{y'}
  \psi( x_1(t_y-\tau_{1}(t_y)))(1-p(t_y)-x_1(t_y))=r(t_y)\psi(x_0(t_y-\tau_{0}(t_y)))(x_1(t_y)-1+q(t_y)).
\end{equation}

We now consider boundary cases for the functions $x_0$ and $x_1$:
\begin{itemize}
    \item If $x_0(t_x) = M_{x_0} = \nu + \epsilon$ or $x_0(t_x) = m_{x_0} = \eta - \epsilon$, then from equation~\eqref{x'} we obtain:
    \begin{align*}
    0 > f_1(t_x)\psi(x_1(t_x - \tau_1(t_x)))(p(t_x) - \nu - \epsilon) &= f_0(t_x)\psi(x_0(t_x - \tau_0(t_x)))(\nu + \epsilon - q(t_x)) > 0, \\
    0 < f_1(t_x)\psi(x_1(t_x - \tau_1(t_x)))(p(t_x) - \eta + \epsilon) &= f_0(t_x)\psi(x_0(t_x - \tau_0(t_x)))(\eta - \epsilon - q(t_x)) < 0,
    \end{align*}
    respectively. Both yield a contradiction.

    \item If $x_1(t_y) = M_{x_1} = 1 - \eta + \epsilon$ or $x_1(t_y) = m_{x_1} = 1 - \mu - \epsilon$, then from equation~\eqref{y'} we obtain:
    \begin{align*}
    0 > f_1(t_y)\psi(x_1(t_y - \tau_1(t_y)))(\eta - \epsilon - p(t_y)) &= f_0(t_y)\psi(x_0(t_y - \tau_0(t_y)))(q(t_y) - \eta + \epsilon) > 0, \\
    0 < f_1(t_y)\psi(x_1(t_y - \tau_1(t_y)))(\mu + \epsilon - p(t_y)) &= f_0(t_y)\psi(x_0(t_y - \tau_0(t_y)))(q(t_y) - \mu - \epsilon) < 0,
    \end{align*}
    respectively. These also lead to contradictions.
\end{itemize}
Finally, consider the case $\lambda = 0$. Then, by the definition of $\mathcal{H}$, we have $(x_0, x_1) \in \mathbb{R}^2 \cap \partial \Omega$ satisfying the  system~\eqref{sist}. This contradicts Lemma~\ref{Fromt}, which ensures that no such solution exists on the boundary. Therefore, $\Sigma \subseteq \Omega$ for all $\lambda \in [0,1]$.
\end{proof}
\subsection{Proof of the main theorems}
\begin{proof}[Proof of Theorem \ref{TeoP}]
By applying Lemma~\ref{cots1}, we conclude that $\mathcal{H}$ is an admissible homotopy. Consequently, the Leray–Schauder degree satisfies
\begin{equation*}
\begin{split}
d_{LS}(I - \mathcal{A}, \Omega, 0) &= d_{LS}(\mathcal{H}[1, \cdot], \Omega, 0) \\
&= d_{LS}(\mathcal{H}[0, \cdot], \Omega, 0).
\end{split}
\end{equation*}
Since $(I - \mathcal{H}[0, \cdot])(\overline{\Omega}) \subseteq \mathbb{R}^2$, it follows that
$$
d_{LS}(\mathcal{H}[0, \cdot], \Omega, 0) = d_B(\mathcal{H}[0, \cdot]_{\overline{\Omega \cap \mathbb{R}^2}}, \Omega \cap \mathbb{R}^2, 0),
$$
where $d_B$ denotes the Brouwer degree.
Using the expression for $\mathcal{H}[0, \cdot]$, we note that
$$
d_B(\mathcal{H}[0, \cdot]_{\overline{\Omega \cap \mathbb{R}^2}}, \Omega \cap \mathbb{R}^2, 0) = d_B(-QN[\cdot]_{\overline{\Omega \cap \mathbb{R}^2}}, \Omega \cap \mathbb{R}^2, 0),
$$
where $QN[x, y] = 0$ corresponds to the algebraic system defined in~\eqref{sist}.
Assume condition \textbf{[S]} holds, and let $(x_i, y_i)$ for $i = 1, \dots, k$ be the solutions of system~\eqref{sist}. By Lemma~\ref{Fromt}, each of these solutions lies in $\Omega$. Then, by the additivity property of the Brouwer degree and condition \textbf{[C]}, we obtain
$$
d_B(-QN[\cdot]_{\overline{\Omega \cap \mathbb{R}^2}}, \Omega \cap \mathbb{R}^2, 0) = -\sum_{i=1}^k \operatorname{sgn} \mathcal{N}(x_i, y_i) \neq 0.
$$
Therefore, by degree theory, the system~\eqref{eq:SistPrincipal} admits at least one $T$-periodic solution in $\Omega$. Finally, considering Proposition~\ref{SolPe} and letting $\epsilon \to 0$, the proof of the theorem is complete.
\end{proof}

\begin{proof}[Proof of Theorem \ref{TeoS}]
It is easy to see that if $0<\eta=\nu<1$ then system \eqref{eq:SistPrincipal} (whit $n=1$) reduces to system
\begin{equation}\label{Sist_simpli}
\begin{array}{lr} 
    \dot x_0(t)& = (\eta - x_0(t)) \Phi_{2}(\mathbf{x})\\
    \dot x_1(t)& = (\eta - x_1(t)) \Phi_{2}(\mathbf{x})
     \end{array}
\end{equation}
from which we have the trivial solution $x_0=x_1=\eta$. 
Now suppose that there is $\mathbf x(t)=(x_0(t),x_1(t))$ is a non-trivial positive $T$-periodic solution of system \eqref{Sist_simpli}, then $x_0$ or $x_1$ is a non-constant function, suppose $x_0$ is a non-trivial positive $T$-periodic function, then there exists $t_{\min},t_{\max}\in[0,T]$ such that $x_0(t_{\min})=\min_{t\in[0,T]} x_0(t)<x_0(t_{\max})=\max_{t\in[0,T]} x_0(t)$ and $$0 = (\eta - x_0(t_{\min})) \Phi_{2}(\mathbf{x}(t_{\min})),\quad 0 = (\eta - x_0(t_{\max})) \Phi_{2}(\mathbf{x}(t_{\max})),$$
but considering that $\Phi_{2}(\mathbf{x}(t))\neq0$ for all $t\in[0,T]$, then $x_0(t_{\min})=x_0(t_{\max})=\eta$, this is a contradiction. Therefore the only positive $T$-periodic solution is the constant solution $x_0=x_1=\eta$.
\end{proof}
\section{Applications in exponential and logistic growth}\label{sec:applications}

Theorem~\ref{TeoP} reduces the problem of establishing the existence of $T$-periodic solutions for the system~\eqref{eq:SistPrincipal} to the analysis of an associated algebraic system. In what follows, we apply this result to two concrete instances arising within the context of the quasispecies model, thereby illustrating the applicability of the general framework to biologically motivated scenarios.

The first case corresponds to exponential growth, where the replication function is given by $\psi(x) = x$. Under this assumption, the system takes the form
\begin{equation}\label{App1}
\dot x_i(t) = \sum_{j=0}^1 f_j(t) Q_{ji}(t) x_j(t-\tau_{j}(t)) - x_i(t) \Phi_{2}(\mathbf{x}), \quad i = 0, 1.
\end{equation}

The second case considers logistic growth, with $\psi(x) = x(1 - x)$, leading to the following system:
\begin{equation}\label{App2}
\dot x_i(t) = \sum_{j=0}^1 f_j(t) Q_{ji}(t) x_j(t - \tau_{1j}(t))(1 - x_j(t - \tau_{2j}(t))) - x_i(t) \Phi_{2}(\mathbf{x}), \quad i = 0, 1.
\end{equation}
It is worth noting that in equation~\eqref{App2}, the results remain valid regardless of whether the delay terms satisfy $\tau_{1j} = \tau_{2j}$ or $\tau_{1j} \neq \tau_{2j}$.
\begin{Corollary}\label{Teoapp1}  
Let $0<\eta<\nu<1$. Then, the problem \eqref{App1} admits at least one positive $T$-periodic solution $(x_0,x_1)$ satisfying
 \begin{equation*}
    \min_{t\in[0,T]}\{Q_{00}(t),Q_{10}(t)\}\leq x_0(t)\leq   \max_{t\in[0,T]}\{Q_{00}(t),Q_{10}(t)\} 
\end{equation*}
\begin{equation*}
   1- \max_{t\in[0,T]}\{Q_{00}(t),Q_{10}(t)\} \leq x_1(t)\leq 1- \min_{t\in[0,T]}\{Q_{00}(t),Q_{10}(t)\}  
\end{equation*}
with $x_0(t)+x_1(t)=1$ for all $t\in[0,T]$. Also if $0<\eta=\nu<1$ then $x_0=x_1=\eta$ is the only positive $T$-periodic solution.
\end{Corollary}

\begin{proof}
The result follows directly from the application of Theorem \ref{TeoP}. The system \eqref{sist} admits a unique solution, and we distinguish two cases:
\begin{itemize}
    \item \textbf{Case 1:} $\overline{f_0} \neq \overline{f_1} $. In this case, the solution is given by:
    \begin{align*}
        x_0 &= \frac{\sqrt{(\overline{f_1} + \overline{pf_1} - \overline{qf_0})^2 + 4 \overline{pf_1} (\overline{f_0} - \overline{f_1})} - (\overline{f_1} + \overline{pf_1} - \overline{qf_0})}{2 (\overline{f_0} - \overline{f_1})}, \\
        x_1 &= 1 - x_0.
    \end{align*}
    It is straightforward to verify that $ x_0 > 0$. Moreover,
    \begin{align*}
        x_0 &= \frac{2 \overline{pf_1}}{\sqrt{(\overline{f_1} + \overline{pf_1} - \overline{qf_0})^2 + 4 \overline{pf_1} (\overline{f_0} - \overline{f_1})} + (\overline{f_1} + \overline{pf_1} - \overline{qf_0})} \\
        &< \frac{2 \overline{pf_1}}{\sqrt{(\overline{pf_1} - (\overline{f_1} - \overline{qf_0}))^2} + (\overline{f_1} + \overline{pf_1} - \overline{qf_0})} \\
        &\leq 1,
    \end{align*}
    thus \( (x_0, x_1) \in\, ]0,1[ \times \,]0,1[ \).
    
    \item \textbf{Case 2:} $\overline{f_0} = \overline{f_1}$. The solution is
    $$
        x_0 = \frac{\overline{pf_1}}{\overline{f_0} + \overline{pf_1} - \overline{qf_0}}, \quad x_1 = \frac{\overline{f_0} - \overline{qf_0}}{\overline{f_0} + \overline{pf_1} - \overline{qf_0}}.
    $$
    Again, it is clear that \( (x_0, x_1) \in\, ]0,1[ \times \,]0,1[\).
\end{itemize}
Therefore, condition \textbf{[S]} is satisfied. Moreover,  since there is a unique solution to \eqref{sist}, condition \textbf{[C]}  is also verified.  The proof of the second part of the corollary follows directly from the application of Theorem \ref{TeoS}.
\end{proof}

\begin{Corollary}\label{Teoapp2}  Let $0<\eta<\nu<1$. 
      Then, problem \eqref{App2} admits at least one positive $T$-periodic solution $(x_0,x_1)$ satisfying
 \begin{equation*}
    \min_{t\in[0,T]}\{Q_{00}(t),Q_{10}(t)\}\leq x_0(t)\leq   \max_{t\in[0,T]}\{Q_{00}(t),Q_{10}(t)\} 
\end{equation*}
\begin{equation*}
   1- \max_{t\in[0,T]}\{Q_{00}(t),Q_{10}(t)\} \leq x_1(t)\leq 1- \min_{t\in[0,T]}\{Q_{00}(t),Q_{10}(t)\}  
\end{equation*}
with $x_0(t)+x_1(t)=1$ for all $t\in[0,T]$. Also if $0<\eta=\nu<1$ then $x_0=x_1=\eta$ is the only positive $T$-periodic solution.
\end{Corollary}

 \begin{proof}
This result also follows from Theorem \ref{TeoP}. The system \eqref{sist} has a unique solution in $]0,1[ \times \,]0,1[$, explicitly given by: 
$$x_0= \frac{\overline{pf_1}+\overline{qf_0}}{\overline{f_0}+\overline{f_1}},\quad  x_1=\frac{\overline{f_0}-\overline{qf_0}+\overline{f_1}-\overline{pf_1}}{\overline{f_0}+\overline{f_1}},$$
It is easy to verify that both components lie in the open interval $]0,1[$. Thus, condition \textbf{[S]} is satisfied. Furthermore, since there is a unique solution to \eqref{sist}, the condition \textbf{[C]} is also verified. The proof of the second part of the corollary follows directly from the application of Theorem \ref{TeoS}.
\end{proof}

\section*{Acknowledgment}

\bibliographystyle{abbrv}
\bibliography{Bib_Delay_Quasispecies}

\end{document}